\documentclass[reqno, 12pt]{amsart}

\usepackage[letterpaper,hmargin=1in,vmargin=1.2in]{geometry}
\usepackage{amsmath, amssymb, amsthm, verbatim, url}
\usepackage{graphicx}
\usepackage{enumerate}
\usepackage{amsmath,amssymb,amsthm,mathrsfs,graphicx,url}
\usepackage[usenames,dvipsnames]{color}
\usepackage[colorlinks=true,linkcolor=Red,citecolor=Green]{hyperref}

\DeclareMathOperator \re {Re}
\DeclareMathOperator \im {Im}

\DeclareMathOperator \supp {supp}

\newtheorem{prop}{Proposition}
\newtheorem{lem}[prop]{Lemma}
\newtheorem*{thm}{Theorem}
\numberwithin{equation}{section}
\numberwithin{prop}{section}
\numberwithin{figure}{section}

\author{Kiril Datchev}
\email{datchev@math.mit.edu}
\address{Department of Mathematics, MIT,
Cambridge, MA 02139, USA}

\title[Quantitative limiting absorption principle in the semiclassical limit] {Quantitative limiting absorption principle in the semiclassical limit}

\begin{document}
\begin{abstract} We give an elementary proof of Burq's  resolvent bounds for  long range semiclassical Schr\"odinger operators. 
Globally, the resolvent norm grows exponentially in the inverse semiclassical parameter, and near infinity it grows linearly. We also weaken the regularity assumptions on the potential.
\end{abstract}
\maketitle

\addtocounter{section}{1}

\noindent Let $\Delta\le 0$ be the Laplacian on $\mathbb R^n$, $n  \ne 2$, and let $E>0$. 
Let
\begin{equation}\label{e:pdef}
P = P_h :=  -h^2 \Delta + V - E, \qquad h >0,
\end{equation}
where, using  polar coordinates $(r,\omega) \in (0,\infty)\times \mathbb S^{n-1}$, we suppose that  $V = V_h(r,\omega)$  and its distributional  derivative $\partial_r V$ are in  $L^\infty ((0,\infty)\times\mathbb S^{n-1})$. Suppose futher  that
 \begin{equation}\label{e:vbound}
 V \le (1+r)^{-\delta_0}, \qquad \partial_r V \le (1+r)^{-1-\delta_0},
 \end{equation}
 for some $\delta_0>0$.
Since $V \in L^\infty(\mathbb R^n)$,  the resolvent $(P-i\varepsilon)^{-1}$ is defined  $L^2(\mathbb R^n) \to H^2(\mathbb R^n)$ for  $\varepsilon >0$ by the Kato--Rellich theorem. 
We prove the following weighted resolvent bounds:

\begin{thm}
For any $s>1/2$ there are $C, R_0, h_0>0$ such that
\begin{equation}\label{e:t1}
\left\| (1+r)^{-s} (P -i\varepsilon)^{-1} (1+r)^{-s}  \right\|_{L^2(\mathbb R^n) \to L^2(\mathbb R^n)}  \le e^{C/h},
\end{equation}
\begin{equation}\label{e:t2}
\left\|(1+r)^{-s}  \mathbf{1}_{\ge R_0} (P-i\varepsilon)^{-1} \mathbf{1}_{\ge R_0}(1+r)^{-s} \right\|_{L^2(\mathbb R^n) \to L^2(\mathbb R^n)} \le C / h,
\end{equation}
 for all $\varepsilon  >0 $, $h \in (0,h_0]$, where $\mathbf{1}_{\ge R_0}$ is the characteristic function of $\{x \in \mathbb R^n \colon |x| \ge R_0\}$.
\end{thm}
This Theorem was first proved by Burq \cite{bu0, bu}, who required $V$ to be smooth, but allowed it  to be a differential operator  on an exterior domain $\mathbb R^n \setminus \overline {\mathcal O}$, $n \ge 1$. Different proofs 
were found by Sj\"ostrand \cite{sres} and Vodev \cite{vo0}. Cardoso and Vodev \cite{cv} gave a version for  manifolds with asymptotically conic or hyperbolic ends, and, most  recently, Rodnianski and Tao \cite{rt} considered Schr\"odinger operators on asymptotically conic manifolds, 
obtaining also bounds for low energies and other refinements.
Here we consider only operators of the form \eqref{e:pdef}, with $n \ne 2$, in order to stress the elementary nature of the proof and to present the ideas in the simplest setting; however, the assumption \eqref{e:vbound} is mild, and our method  should also give simplifications and low regularity results in more general cases.

Our proof is closest in spirit to that of Cardoso and Vodev \cite{cv} (see also \cite{voasy, vo}). The novelty is a \textit{global} Carleman estimate of the form
\[
\left\| (1+r)^{-s} e^{\varphi/h}v\right\|^2_{L^2(\mathbb R^n)}  \le \frac {C}{h^2}\left\|(1+r)^{s}e^{\varphi/h}(P-i\varepsilon)v\right\|^2_{L^2(\mathbb R^n)} + \frac {C\varepsilon}h \|e^{\varphi/h}v\|^2_{L^2(\mathbb R^n)},
\]
with  $C$ independent of the support of $v$, and with $\varphi = \varphi(r)$  nondecreasing and constant outside of a compact set: see Lemma \ref{l:carl}. 
Carleman estimates are crucial in all the proofs mentioned above, and one nice feature of our approach is that in this setting the construction of $\varphi$ is relatively simple: see Lemma~\ref{l:weights}.

The $h$ dependence in \eqref{e:t1} is optimal in general, but improvements hold under dynamical assumptions on the Hamilton flow $\Phi(t) = \exp t(2 \xi \partial_x - \partial_x V(x)  \partial_\xi) $ on $T^*\mathbb R^n$. (Note, however, that  $\Phi$ may be undefined under our regularity assumptions.) See \cite{wu} for a recent survey, and \cite{dy, nz,ch} for more recent results in this active area.  For example,  if $\Phi$ is \textit{nontrapping}  at energy $E$ (e.g. if $V \equiv 0$), then \eqref{e:t1} can be replaced by
\begin{equation}\label{e:nontrap}
\left\| (1+r)^{-s} (P -i\varepsilon)^{-1} (1+r)^{-s}  \right\|_{L^2(\mathbb R^n) \to L^2(\mathbb R^n)}  \le C/h.
\end{equation}
In this sense \eqref{e:t2} says that applying $\mathbf{1}_{\ge R_0}$ cutoffs removes the loss exhibited by \eqref{e:t1} compared to \eqref{e:nontrap}. 
It would be interesting to know if some improvement over \eqref{e:t1} persists if we remove one of the $\mathbf{1}_{\ge R_0}$ factors from \eqref{e:t2}, and if $\mathbf{1}_{\ge R_0}$ can be replaced by a finer cutoff; for some  results in this direction, see \cite{dv, rt, hv}.  For example, in \cite{dv}, Vasy and I show that if $\Phi$ is  `mildly'  trapping then \eqref{e:t2} holds with $\mathbf{1}_{\ge R_0}$ replaced by a microlocal cutoff vanishing only on an arbitrarily small neighborhood of the trapped set.

In \cite{voasy, vo}, Vodev studied operators of the form \eqref{e:pdef},  satisfying \eqref{e:vbound}, but with  $V$ replaced by $h^\nu V$ for some $\nu>0$; he showed that in that case the bound \eqref{e:nontrap} holds. He also allowed $V$ to contain a magnetic term and a less regular short range term. 

I am grateful to Maciej Zworski for encouraging me to write this note, and for many very helpful discussions and suggestions. Thanks also to Nicolas Burq for  pointing out a problem with an earlier version of this argument.
Thanks finally to Georgi Vodev for sharing his preprint \cite{vo}, which gave me the initial idea for the proof.  I am also grateful for the support of a National Science Foundation postdoctoral fellowship.

\section{Proof of Theorem}

\noindent We begin with two lemmas; the first constructs a nondecreasing Carleman weight for $P$ which is constant outside of a compact set, and the second uses this weight to prove a global Carleman estimate. Without loss of generality, we  assume $0<2s-1<\delta_0<1$.  Put
\[
\delta  :=  2s - 1<\delta_0, \qquad w = w_\delta(r) := 1 - (1+r)^{-\delta},\qquad m :=   (1+r^2)^{(1+\delta)/4}.
\]
  
\begin{lem}\label{l:weights} If $\delta>0$ is small enough, there are $h_0$, $R_0>0$, and  $\varphi = \varphi(r) \in C^\infty([0,\infty))$ with 
$\varphi' \ge 0$ and $ \supp \varphi' = [0,R_0]$, such that
\begin{equation}\label{e:lweights}
  \partial_r \left(w(r)(E-V_h(r,\omega) + \varphi'(r)^2 - h \varphi''(r))\right) \ge Ew'(r)/4,
\end{equation}
for all $h \in (0,h_0]$, $r>0$, $\omega \in \mathbb S^{n-1}$.
\end{lem}

\begin{lem}\label{l:carl} Let  $\delta$, $h_0$, and $\varphi = \varphi(r)$ be as in Lemma \ref{l:weights}. There is $C >0$ such that
\begin{equation}\label{e:lcarl}
\left\| m^{-1} e^{\varphi/h}v\right\|^2_{L^2(\mathbb R^n)}  \le \frac {C}{h^2}\left\|me^{\varphi/h}(P-i\varepsilon)v\right\|^2_{L^2(\mathbb R^n)} + \frac {C\varepsilon}h \|e^{\varphi/h}v\|^2_{L^2(\mathbb R^n)},
\end{equation}
for all $v \in C_0^\infty(\mathbb R^n)$, $\varepsilon  \ge0$, and $h \in (0,h_0]$.
\end{lem}

\begin{proof}[Proof of Lemma \ref{l:weights}] 
For  $B, \, R, \, R_0$  (depending on $\delta$) to be determined later, put
\[
\psi = \psi_\delta(r) := \begin{cases}  \delta_0^{-1} , &r \le R, \\  \frac B {w(r)} - \frac E 4 , & R < r < R_0, \\ 0, &r \ge R_0,\end{cases}
\]
We will show  that, for $\delta$  small enough, there are $B, \, R, \, R_0$ which make $\psi$ continuous and
\begin{equation}\label{e:psi0}
- E/2 \le    \psi -  V -  (\partial_r V - \psi')w/w', \qquad r>0, \ r \ne R, \  r \ne R_0.
\end{equation}
Suppose for a moment that this is done. Fix $\rho \in C_0^\infty((0,\infty))$ with $\rho \ge 0$, $\int \rho = 1$, and for $\eta>0$, put $\rho_\eta(r) = \rho(r/\eta)/\eta$. If $\eta$ and $h_0$ are sufficiently small, then  we may take 
\[
\varphi(r) := \int_0^r \widetilde \psi(t)dt, \qquad \widetilde \psi := \rho_\eta * \sqrt{\psi}.
\]
It remains to find  $B$, $R$, and $R_0$ such that $\psi$ is continuous and satisfies \eqref{e:psi0}. 

Note that, by \eqref{e:vbound}  we have
\[
V + (\partial_r V) w/w' \le G_\delta(r) :=  (1+r)^{-\delta_0} + \delta^{-1}(1 - (1+r)^{-\delta})(1+r)^{\delta-\delta_0},
\]
and
\[
G'_\delta(r) = (\delta^{-1} - 1)\delta_0(1+r)^{-1-\delta_0} -\delta^{-1}(\delta_0-\delta)(1+r)^{-1-\delta_0+ \delta}.
\]
So, for each $\delta \in (0,\delta_0)$, $G_\delta$ attains its maximum value at $r_{\max}$ which is given by
\begin{equation}\label{e:gmax}
(1+r_{\max})^\delta :=(1-\delta)/(1-\delta/\delta_0) = 1 + \delta(\delta_0^{-1} - 1) + O(\delta^2).
\end{equation}
Hence we have, for all $r>0$,
\[\begin{split} 
G_\delta(r) &= (1+r)^{-\delta_0} \left(1 - \delta^{-1} + \delta^{-1}(1+r)^\delta \right) \\&\le \left(1 - \delta(\delta_0^{-1} - 1) + O(\delta^2)\right)^{\delta_0/\delta} \delta_0^{-1}(1+O(\delta)).
\end{split}\]
Since $(1-x)^{1/x} \le 1/e$ for $x > 0$ and since $\delta_0<1$, this implies, for $\delta$ small enough,
\[
G_\delta(r) \le  e^{-(1-\delta_0) + O(\delta)} \delta_0^{-1}(1+O(\delta)) \le \delta_0^{-1}.
\]
Consequently, regardless of the value of $R$, we  have, for $r < R$,
\begin{equation}\label{e:rsmall}
 \psi -  V -  (\partial_r V - \psi')w/w' 
 \ge \delta_0^{-1}  - G_\delta(r) \ge 0,
\end{equation}
which implies \eqref{e:psi0} for $r<R$.
We will take $R>0$ large enough that
\begin{equation}\label{e:r1}
r \ge R \Longrightarrow G_\delta(r) \le E/4.
\end{equation}
First let us see that \eqref{e:r1} 
implies \eqref{e:psi0} for $r>R$. Indeed, for $r > R_0$, \eqref{e:r1} implies
\[
\psi_0 -  V - (\partial_r V - \psi'_0)w/w'  =   -  V - (\partial_r V)w/w'  \ge - G_\delta(r) \ge -E/4,
\]
while, if $R < r < R_0$, we have $ \psi_0 + \frac w {w'} \psi_0' = -E  /4$, and hence \eqref{e:r1} implies
\[
\psi_0 -  V -  (\partial_r V - \psi'_0)w/w'\ge - E/4 - G_\delta(r)   \ge -E/2.
\]
Next  note that, for any $R>0$,  $\psi$ is continuous if and only if we take $B$ and $R_0$ such that
\[
B = (\delta_0^{-1} + E/4)w(R), \qquad w(R_0) = 4B/E = (1 + 4 \delta_0^{-1} E^{-1})w(R).
\]
Since $w$ takes values strictly between $0$ and $1$, this is possible if and only if
\begin{equation}\label{e:r2}
w(R) < 1/(1+4\delta_0^{-1}E^{-1}).
\end{equation}
Consequently, to complete the construction, it suffices to show that, if $\delta$ is small enough, then there is $R>0$ such that \eqref{e:r1} and \eqref{e:r2} hold. Define $R$ by
\[
(1+R)^{\delta-\delta_0} := \delta E/4,
\]
so that
\[
G_\delta(R) \le \delta^{-1} (1+R)^{\delta-\delta_0} = E/4.
\]
Note that, for $\delta>0$ sufficiently small we have, by \eqref{e:gmax},
\[
(1+R)^\delta = (\delta E/4)^{-\delta/(\delta_0-\delta)} = 1 + \delta_0^{-1} \delta |\ln \delta| + O(\delta) > (1+ r_{\max})^\delta.
\]
So $G_\delta'(R)<0$ for $r \ge R$, and we have \eqref{e:r1}. 
Similarly,
\[
w(R) = 1 - (1 + R)^{-\delta} = O(\delta|\ln\delta|),
\]
so this choice of $R$ also gives \eqref{e:r2}  for $\delta>0$ sufficiently small, as desired.
\end{proof}

\begin{proof}[Proof of Lemma \ref{l:carl}] 
Let
\[\begin{split}
P_\varphi :=\, &e^{\varphi/h} r^{(n-1)/2} (P  - i \varepsilon) r^{-(n-1)/2} e^{-\varphi/h} \\= &- h^2 \partial_r^2  + 2h \varphi' \partial_r + \Lambda + V_\varphi - E - i \varepsilon, 
\end{split}\]
where
\[\begin{split}
0 \le \Lambda &:=\begin{cases}0, & n=1, \\ h^2 r^{-2} \left(-\Delta_{\mathbb S^{n-1}} + (n-1)(n-3)/4\right), & n\ge 3,\end{cases} \\
V_\varphi &:= V - \varphi'^2 + h\varphi''.
\end{split}\]
Let $\int_{r,\omega}$  denote the integral over  $(0,\infty) \times \mathbb S^{n-1}$ with respect to  $drd\omega$, where $d\omega$ is the usual measure on the unit sphere $\mathbb S^{n-1}$.  Then \eqref{e:lcarl} is equivalent to
\begin{equation}\label{e:carpphi}
\int_{r,\omega} w'  |u|^2 \le \frac {C}{h^2} \int_{r,\omega} \frac{|P_\varphi u|^2}{w'}  + \frac {C\varepsilon}h \int_{r,\omega}  |u|^2, \qquad u \in e^{\varphi/h} r^{(n-1)/2}C_0^\infty(\mathbb R^n).
\end{equation}
We may assume $\varepsilon \le h$, since $w' \le1$ makes \eqref{e:carpphi} trivial for $\varepsilon > h $.
We will prove
\begin{equation}\label{e:careps}
\int_{r,\omega} \partial_r \left(w(E-V_\varphi)\right) |u|^2    \le \frac {2}{h^2} \int_{r,\omega} \frac{|P_\varphi u|^2}{w'}  + \frac {C\varepsilon}h \int_{r,\omega}  |u|^2,
\end{equation}
which, together with \eqref{e:lweights}, implies \eqref{e:carpphi}.
In the spirit of \cite{cv, rt, voasy, vo}, put
\[
F(r) := \| h\partial_r u(r,\omega)\|^2_S - \langle (\Lambda + V_\varphi(r,\omega)  - E)u(r,\omega),u(r,\omega)\rangle_S, \qquad r>0,
\]
where $\| \cdot \|_S$ and $\langle \cdot, \cdot \rangle_S$ are the norm and inner product in $L^2(\mathbb S^{n-1})$. 
Note that 
\begin{equation}\label{e:wfint}
\int_0^\infty (w(r)F(r))' dr \le - \lim_{r \to 0} w(r) \liminf_{r \to 0} F(r) = 0. 
\end{equation}
We use the selfadjointness of $\Lambda + V_\varphi - E$ to compute the derivative of $F$ in terms of $P_\varphi$:
\[\begin{split}
F'   &= 2  \re \langle h^2u'',u'\rangle_S  - 2 \re \langle(\Lambda + V_\varphi - E)u,u' \rangle_S +  2r^{-1} \langle \Lambda u,u\rangle_S - \langle V_\varphi' u,u\rangle_S\\
&= - 2 \re \langle P_\varphi u, u'\rangle_S + 4 h\varphi' \|   u' \|_S^2 + 2\varepsilon  \im \langle u,  u' \rangle_S  + 2r^{-1} \langle \Lambda u,u\rangle_S - \langle  V_\varphi' u,u\rangle_S, \\
\end{split}\]
where $u': =\partial_r u$ and $V_\varphi' : = \partial_r V_\varphi$. Consequently
\[\begin{split}
w F' + w' F = &- 2w \re \langle P_\varphi u,  u' \rangle_S + \left(4h^{-1} w \varphi' + w' \right)\|h u'\|_S^2  + 2w\varepsilon \im \langle u, u'\rangle_S \\
& + \left(2wr^{-1} - w'\right) \langle \Lambda u,u\rangle_S + \langle  \left(w(E-V_\varphi)\right)'u,u\rangle_S .
\end{split}\]
Using $w \varphi'  \ge 0$, $w'>0$, $\Lambda \ge 0$, $2wr^{-1} - w'>0$, and  $ - 2 \re \langle a,   b \rangle +  \|b\|^2 \ge -\|a\|^2 $, we obtain
\[ \begin{split}
w F' + w' F \ge & -  \frac{w^2}{h^2w'} \|P_\varphi u\|_S^2 + 2w\varepsilon \im \langle u, u'\rangle_S +  \langle\left(w(E-V_\varphi)\right)'u,u\rangle_S.
\end{split}\]
Combining this with \eqref{e:wfint} and using $w \le 1$  gives
 \begin{equation}\label{e:epsrem1}\begin{split}
\int_{r,\omega} \left(w(E-V_\varphi)\right)' |u|^2  \le 
\frac {1}{h^2} \int_{r,\omega} \frac{|P_\varphi u|^2}{w'}  + & 2 \varepsilon \int_{r,\omega} |u  u'| . 
\end{split}\end{equation}
On the other hand, for all $\gamma>0$ there is $C_\gamma$ such that
 \[\begin{split}
 \int_{r,\omega} |hu'|^2 &= \re \int_{r,\omega} \bar u ( P_\varphi -  2h\varphi'\partial_r - \Lambda - V_\varphi + E + i \varepsilon)u 
\\ &\le \int_{r,\omega} |P_\varphi u| |u|   + 2    \int_{r,\omega}  \varphi' |h  u'|  |u|  +   \int_{r,\omega}  |E - V_\varphi| |u|^2 
\\ & \le \int_{r,\omega} |P_\varphi u|^2  +   C_\gamma \int_{r,\omega}  |u|^2  + \gamma  \int_{r,\omega} \varphi' |h u'|^2.
\end{split}\]
Choosing $\gamma = 1/(2\max \varphi')$ gives
 \begin{equation}\label{e:epsrem2}\begin{split}
 \int_{r,\omega} |hu'|^2 \le 2\int_{r,\omega} |P_\varphi u|^2  +   C \int_{r,\omega}  |u|^2 .
\end{split}\end{equation}
Applying $2  \int_{r,\omega}  |u u'|  \le h^{-1}\int_{r,\omega} |u|^2 + h^{-1} \int_{r,\omega}  |hu'|^2 $ to \eqref{e:epsrem1}, and using  \eqref{e:epsrem2} and  $\varepsilon \le h$, gives \eqref{e:careps}.
\end{proof}

\begin{proof}[Proof of Theorem] 
Put $C_0 = 2 \max \varphi$. Then, since $\varphi(r) = C_0$ for $r \ge R_0$, \eqref{e:lcarl} implies
\[\begin{split}
e^{-C_0/h}\left\| m^{-1} \mathbf{1}_{\le R_0} v\right\|^2_{L^2}  +  \left\| m^{-1}  \mathbf{1}_{\ge R_0} v\right\|^2_{L^2}  & \le e^{-C_0/h}\left\| m^{-1}   e^{\varphi/h}v\right\|^2_{L^2} \\
&\le  \frac {C}{h^2}\left\|m(P-i\varepsilon)v\right\|^2_{L^2} + \frac {C_1\varepsilon}h \|v\|^2_{L^2},
\end{split}\]
where we abbreviated $L^2(\mathbb R^n)$ as $L^2$.
Then using
\[\begin{split}
2 \varepsilon \|v\|^2_{L^2} = & -2 \im \langle (P- i\varepsilon)v,v\rangle_{L^2} \le  \gamma^{-1} \left\|m \mathbf{1}_{\ge R_0} (P  - i \varepsilon) v\right\|^2_{L^2} + \\  & \gamma \|m^{-1} \mathbf{1}_{\ge R_0} v\|^2_{L^2} +  \gamma_0^{-1} \left\| m\mathbf{1}_{\le R_0} (P  - i \varepsilon) v\right\|^2_{L^2} +    \gamma_0 \|m^{-1} \mathbf{1}_{\le R_0} v\|^2_{L^2},
\end{split}\]
with $\gamma = e^{-2C_0/h}$  and   $\gamma_0 = h/C_1$  we  conclude that, for $h$ sufficiently small,
\begin{equation}\label{e:opest}\begin{split}
e^{-C/h}\left\|  m^{-1} \mathbf{1}_{\le R_0} v\right\|^2_{L^2}  +  \left\| m^{-1}  \mathbf{1}_{\ge R_0} v\right\|^2_{L^2} &\le \\   e^{C/h} \left\|m\mathbf{1}_{\le R_0} (P  - i \varepsilon) v\right\|^2_{L^2}  &+  \frac {C}{h^2}\left\|m \mathbf{1}_{\ge R_0}(P-i\varepsilon)v\right\|^2_{L^2},
\end{split}\end{equation}
for all $v \in C_0^\infty(\mathbb R^n)$. We will deduce from \eqref{e:opest} that, for any $f \in L^2$, we have
\begin{equation}\label{e:resest}\begin{split}
e^{-C/h}\left\|  \mathbf{1}_{\le R_0} (P-i\varepsilon)^{-1} m^{-1} f\right\|^2_{L^2}  +  \left\| m^{-1}  \mathbf{1}_{\ge R_0} (P-i\varepsilon)^{-1} m^{-1} f\right\|^2_{L^2} &\le   \\   e^{C/h} \left\|\mathbf{1}_{\le R_0} f \right\|^2_{L^2}  &+  \frac {C}{h^2}\left\|\mathbf{1}_{\ge R_0}f \right\|^2_{L^2},
\end{split}\end{equation}
from which the Theorem follows.
For this we need the fact that, for fixed $\varepsilon, h>0$,
\begin{equation}\label{e:bothbounded}
\frac 1 {C_{\varepsilon, h}} \|m v\|_{H^2}\le  \|m(P-i\varepsilon)v\|_{L^2} \le C_{\varepsilon, h} \|mv\|_{H^2}, \qquad mv \in H^2.
\end{equation}
Momentarily assuming \eqref{e:bothbounded}, fix  $f \in L^2$, so   $m(P-i\varepsilon)^{-1} m^{-1} f \in H^2$. Take $v_k \in C_0^\infty$ with
\[
\|mv_k - m(P-i\varepsilon)^{-1} m^{-1} f\|_{H^2} \to 0 \textrm{ as } k \to \infty.
\]
Then in particular $\|m^{-1} v_k - m^{-1}  (P-i\varepsilon)^{-1} m^{-1}f\|_{L^2} \to 0$, and, by \eqref{e:bothbounded}, 
\[
\|m(P - i \varepsilon) v_k  -  f\|_{L^2}  \le C_{\varepsilon, h} \|mv_k - m(P-i\varepsilon)^{-1} m^{-1} f\|_{H^2} \to 0 \textrm{ as } k \to \infty.
\]
Consequently  \eqref{e:resest} follows by applying \eqref{e:opest} wtih $v_k$  in place of $v$, and letting $k \to \infty$.

It remains to prove \eqref{e:bothbounded}. Below, $a \lesssim b$ means $a \le C b$ with $C$ depending on $\varepsilon$ and $h$ (but not $v$). By the Kato--Rellich Theorem, $(P-i\varepsilon)^{-1}$ is bounded $L^2 \to H^2$, so
\begin{equation}\label{e:kr}
 \|mv\|_{H^2} \lesssim  \|(P-i\varepsilon) mv\|_{L^2} \lesssim  \|mv\|_{H^2},
\end{equation}
for all $v$ with $m v \in H^2$. Meanwhile,
$[P,m] = -2h^2 m' \partial_r -h^2m''-h^2(n-1)m'/r$ is bounded  $H^2 \to L^2$,
  allowing us to deduce the second of \eqref{e:bothbounded} from the second of \eqref{e:kr}:
 \[
 \| m(P-i\varepsilon)v\|_{L^2} \lesssim \|mv\|_{H^2} +  \|[P,m]v\|_{L^2} \lesssim  \|mv\|_{H^2}.
\]
Similarly we deduce the first of  \eqref{e:bothbounded} from the first of \eqref{e:kr}:
\[
\|m v\|_{H^2} \lesssim \|m(P-i\varepsilon)v\|_{L^2} +  \|[P,m]v\|_{L^2}  \lesssim \|m(P-i\varepsilon)v\|_{L^2}.
\]
\end{proof}

\def\arXiv#1{\href{http://arxiv.org/abs/#1}{arXiv:#1}}

\end{document}